\documentclass[11pt,a4paper]{article}
\usepackage[utf8]{inputenc}
\usepackage[T1]{fontenc}
\usepackage{lmodern}
\usepackage[margin=2.35cm]{geometry}
\usepackage{microtype}
\usepackage{amsmath,amssymb,mathtools,amsthm}
\usepackage{graphicx}
\usepackage{booktabs}
\usepackage{array}
\usepackage{url}
\usepackage{authblk}
\usepackage{xcolor}
\usepackage{tikz}
\usepackage{enumitem}
\usepackage{placeins}
\usepackage[hidelinks]{hyperref}
\usetikzlibrary{arrows.meta,positioning,shapes.geometric,calc,fit}
\graphicspath{{figures/}}

\newcommand{\E}{\mathbb{E}}
\newcommand{\dif}{\,\mathrm{d}}
\newcommand{\MS}{\mathrm{MS}}
\newcommand{\MD}{\mathrm{MD}}

\newtheorem{proposition}{Proposition}
\newtheorem{theorem}{Theorem}

\newcounter{algctr}
\newenvironment{algorithmblock}[1]{%
  \refstepcounter{algctr}\begin{figure}[t]\small\hrule\vspace{3pt}%
  \noindent\textbf{Algorithm \thealgctr: #1}\par\vspace{2pt}%
}{\vspace{2pt}\hrule\end{figure}}

\title{{\bf Stochastic inflows and seasonal water value at Manantali}}

\author[1,2]{Steeven B. Affognon\thanks{Corresponding author: \texttt{belvinos@gmail.com}}}
\author[1]{Babacar M. Ndiaye}
\author[1]{Pierre Mendy}
\author[3,4]{Cheikh M. F. Kebe}
\author[5]{Sandrine Charousset}
\author[6]{Omar Goudiaby}
\author[7]{Thomas Ouillon}

\affil[1]{Laboratory Mathematics of Decision and Numerical Analysis, Cheikh Anta Diop University of Dakar, Dakar, Senegal}
\affil[2]{Department of Mathematics, University of Nairobi, Nairobi, Kenya}
\affil[3]{Laboratoire Eau, Energie, Environnement et proc\'ed\'es Industriels (LE3PI), Universit\'e Cheikh Anta Diop de Dakar (UCAD), Dakar, Senegal}
\affil[4]{Centre de Tests des Syst\`emes Solaires (CT2S), Dakar, Senegal}
\affil[5,7]{EDF Lab Paris-Saclay, Département Optimisation Simulation Risques
Statistiques pour les marchés de l'énergie (OSIRIS), Palaiseau, France}
\affil[6]{Laboratoire Dynamique des Territoires et Développement (Leïdi), Université Gaston Berger (UGB), BP 2, Saint-Louis, Senegal}

\date{}

\begin{document}
\maketitle

\begin{abstract}
Seasonal hydropower operation requires release decisions that account for uncertain
future inflows and the opportunity cost of stored water. We assess the Modelling
and Optimisation of Stochastic Seasonal Hydropower Storage (MOSSHOOS) framework
for the Manantali reservoir using 660 monthly energy-inflow observations from
1961--2015. The framework combines multiscale inflow representation, blocked
resampling and scenario reduction with stochastic dual dynamic programming, and
uses a seasonal square-root diffusion and Hamilton--Jacobi--Bellman formulation
as a continuous-time interpretation of marginal water value. Scenario reduction
improves the energy score by 47.6\% relative to independent monthly sampling,
but the marginal distribution is rejected by a Kolmogorov--Smirnov test and the
lower-tail p10 error remains 103.9\%, indicating inadequate dry-tail calibration.
Consistent with that limitation, 50 paired Monte Carlo evaluations show no
detectable cost or deficit advantage of the stochastic policy over its
deterministic-equivalent counterpart at the archived operating configuration.
As an independent benchmark, the repaired open-source \texttt{plan4res} seasonal
storage valuation chain yields a dominant marginal value of 81.66~USD/MWh,
within 3\% of the MOSSHOOS mean of 79.25~USD/MWh. These results support a
coherent stochastic inflow-to-water-value methodology while showing that
hydrological tail representation, rather than economic scaling, is the principal
limitation to address before operational application.
\end{abstract}


\section{Introduction}\label{sec:intro}
Hydropower storage is an intertemporal allocation problem under uncertainty. A
reservoir operator must decide how much water to release now, when generation has
an immediate economic value, and how much to retain for future periods in which
inflow may be lower, demand may be higher, or alternative generation may be more
expensive. The shadow price associated with that trade-off is the \emph{water
value}: the marginal reduction in expected future operating cost obtained from an
additional unit of stored water \cite{labadie2004,gjelsvik2010,dequeiroz2016}. Water value is
therefore not a fixed tariff and not a purely hydraulic quantity. It is an
endogenous stochastic opportunity cost shaped by the inflow law, storage and
release constraints, demand, generation costs, terminal conditions, and the
information available when a decision is taken.

That dependence on the inflow law is especially important for seasonal reservoirs.
A model can reproduce the annual mean hydrograph and still be inadequate for
operation if it misses persistence of dry spells, the timing of the wet-season
transition, serial dependence between months, or the lower and upper tails that
drive shortage and spill. Scenario quality must consequently be judged by more
than a visual fit to the seasonal mean. In stochastic programming, reduced
scenario sets should preserve features that are relevant to the objective and to
future recourse decisions \cite{heitsch2003,kaut2007}; proper multivariate scores
provide one summary of distributional quality \cite{gneiting2007}, while marginal,
autocorrelation, and tail diagnostics reveal failure modes that a single score can
hide. Block resampling is attractive when temporal dependence must be retained
without imposing a fully parametric residual law \cite{politis1994}.

The Senegal River setting makes this multiscale issue concrete. Wavelet methods
are well established for separating nonstationary hydrological variability across
temporal scales \cite{labat2005,sang2013}. The basin itself has undergone marked
hydroclimatic shifts, while Manantali operation couples hydropower production to
flood and downstream water-management objectives \cite{bodian2020,bruckmann2022,raso2020}.
The hydrological record used in the MOSSHOOS project exhibits a pronounced annual
cycle together with substantial interannual variability. Our companion analysis
of daily inflows shows that annual and slower hydrological scales are statistically
distinguishable and should not be conflated in a single notion of ``seasonality''
\cite{affognon2026wavelet}. For the monthly Manantali archive considered here,
this motivates a deliberately transparent construction: a coarse Haar separation,
a Fourier seasonal target, blocked residual resampling, and reduction to a finite
weighted scenario law. Importantly, the present paper does not assume that this
construction is adequate merely because it preserves the seasonal peak. Its lower
tail, serial dependence, and out-of-sample control consequences are tested
explicitly.

A second representation of inflow is useful for a different reason. Rather than
replacing the finite scenario law used by stochastic programming, a continuous
stochastic differential equation provides a low-dimensional state description on
which the reservoir control problem can be characterized analytically. We use a
seasonal square-root diffusion of Cox--Ingersoll--Ross type, whose mean-reverting
structure is compatible with a time-varying seasonal target and whose diffusion
coefficient is state dependent \cite{cir1985}. The classical Feller boundary
criterion \cite{feller1951} then becomes operationally relevant: if the seasonal
ratio drops below its strict-positivity threshold during the dry part of the year,
the zero-inflow boundary cannot simply be ignored by the numerical HJB solver.
This distinction between nonnegativity and uniform strict positivity is important
for interpreting the continuous model and for selecting a positivity-preserving
numerical treatment \cite{alfonsi2005,lord2010}.

The storage decision itself can be described in two complementary ways. In
discrete time, stochastic dual dynamic programming (SDDP) approximates the convex
future-cost function by supporting hyperplanes generated from stage-problem dual
information \cite{pereira1991,philpott2008,shapiro2011,fullner2025}. Risk-aware
extensions have been developed specifically for multistage hydrothermal operation
\cite{philpottmatos2012}, and modern open implementations have made SDDP algorithms
more transparent and reproducible \cite{dowson2021}. This construction is well
suited to multistage hydrothermal problems because it propagates the future value
of storage without enumerating a full scenario tree. In continuous time, dynamic
programming leads to a Hamilton--Jacobi--Bellman (HJB) equation whose storage
derivative is the local marginal value of stored water
\cite{fleming2006,kushner2001,crandall1992,soner1986}. These are not the same numerical algorithm, and
there is no reason for their stage-by-stage trajectories to coincide when time
resolution, information structure, scenario aggregation, or terminal values
differ. They do, however, encode the same economic comparison: release should
continue while the immediate marginal system benefit of water exceeds its
continuation value. Our recent InFlow preprint develops this SDDP--HJB
cross-certification principle in a normalized risk-aware benchmark
\cite{affognon2026inflow}; here we use the idea more narrowly to interpret the
archived Manantali results.

This paper therefore focuses on the entire evidence chain from stochastic inflow
representation to marginal storage value. That chain is often where apparently
successful hydro-optimization studies become difficult to audit: a scenario
generator may be assessed separately from the optimizer, the optimizer may be
reported without direct interpretation of its dual variables, and a continuous
control model may be presented without a like-for-like economic benchmark. Our
objective is not to introduce yet another optimizer in isolation. It is to ask
whether the archived hydrology generates a credible stochastic law, how that law
enters discrete and continuous control formulations, whether the resulting water
values have a consistent economic interpretation, and which conclusions survive
when raw run-level evidence is distinguished from report-only summaries.

The open-source \texttt{plan4res} framework provides an independent benchmark for
that purpose. Its Seasonal Storage Valuation (SSV) layer exchanges Bellman
functions and water values with simulation and capacity-expansion modules
\cite{beulertz2019}. In the tested open-source configuration, the SSV path failed
before the first backward pass because the required LP relaxation was not activated
in the SMS++ interface. The missing dual is exactly the sensitivity needed for a
Benders cut, so repairing that path is mathematically relevant to the benchmark.
Nevertheless, the software repair is intentionally not the scientific center of
the paper. The main object is the MOSSHOOS stochastic inflow-to-water-value
framework; \texttt{plan4res} is used to test whether an independently computed
marginal-value scale is economically coherent.

The Manantali application is deliberately limited to a single-reservoir benchmark.
An operational representation of the wider OMVS system would have to account for
hydraulic coupling among Manantali, Gouina and F\'elou, downstream regulation and
multi-purpose constraints associated with Diama, travel times, head-dependent
conversion, environmental releases, and turbine limits. Those processes are not
implicitly absorbed into the present parameters. Likewise, the current economic
benchmark values water against a reference generation fleet rather than a joint
stochastic residual-load model with high solar and wind penetration. These
limitations define extensions of the framework rather than hidden assumptions of
the present results.

Within that scope, the study makes five linked contributions. First, it makes the
stochastic inflow layer explicit by separating observed seasonality, blocked
scenario construction, scenario reduction, and a continuous seasonal square-root
diffusion. Second, it evaluates the scenario law using joint, marginal,
autocorrelation, and tail diagnostics and carries those diagnostics into the
interpretation of policy results. Third, it derives in detail why an active SDDP
cut slope and the negative HJB storage derivative represent the same marginal
water-value concept, and it obtains the corresponding release threshold from the
HJB Hamiltonian. Fourth, it uses the repaired \texttt{plan4res} SSV chain as an
independent benchmark without imposing a false temporal alignment between its 18
stages and the 12 MOSSHOOS months. Fifth, it reanalyzes paired Monte Carlo and
storage-sensitivity evidence so that the principal run-level results are distinguished
from the additional high-storage sensitivity reported in the archived project summary. The resulting emphasis is intentional:
stochastic hydrology and seasonal water valuation are the scientific core;
software repair is supporting evidence for independent verification.

\section{Stochastic valuation framework}\label{sec:methods}

\subsection{Data, scope, and evidence hierarchy}
The hydrological archive contains 660 monthly Manantali-equivalent energy inflows
from 1961--2015. The series has mean 102.63~GWh/month, standard deviation
139.7~GWh/month, coefficient of variation 1.36, and observed range
0--785.98~GWh/month. The analysis is single-reservoir: it does not yet represent
travel time, head variation, environmental and multi-purpose releases, or the
full Manantali--Gouina--F\'elou--Diama system.

Numerical claims follow an evidence hierarchy in which raw paired-run tables and
machine-readable result files take precedence over rounded report summaries.
Table~\ref{tab:evidence} lists the scientific evidence objects used in the analysis. Internal filenames are retained only in the accompanying reproducibility archive.

\begin{table}[t]
\centering
\caption{Evidence objects used in the analysis.}
\label{tab:evidence}
\begin{tabular}{p{0.32\textwidth}p{0.59\textwidth}}
\toprule
Evidence object & Scientific use \\
\midrule
Monthly decomposition table & 660 observed inflows, Haar coarse component, Fourier seasonal target \\
Reduced-scenario table & ten weighted annual inflow scenarios \\
Scenario-validation table & energy score, marginal KS statistic, autocorrelation error, p10/p90 errors \\
Paired Monte Carlo tables & stochastic-policy and deterministic-equivalent outcomes \\
Galerkin convergence table & $L^2$ error versus Haar basis size \\
Water-value benchmark file & \texttt{plan4res} cut slopes and MOSSHOOS marginal values \\
Archived storage-sweep summary & high-storage sensitivity analysis \\
\bottomrule
\end{tabular}
\end{table}

\subsection{Scientific architecture}
Figure~\ref{fig:workflow} places the stochastic hydrology and storage-control
pipeline at the center. The \texttt{plan4res} branch is intentionally drawn as a
parallel benchmark rather than as the parent of MOSSHOOS. The layout also removes
the crossing arrows of the previous figure.

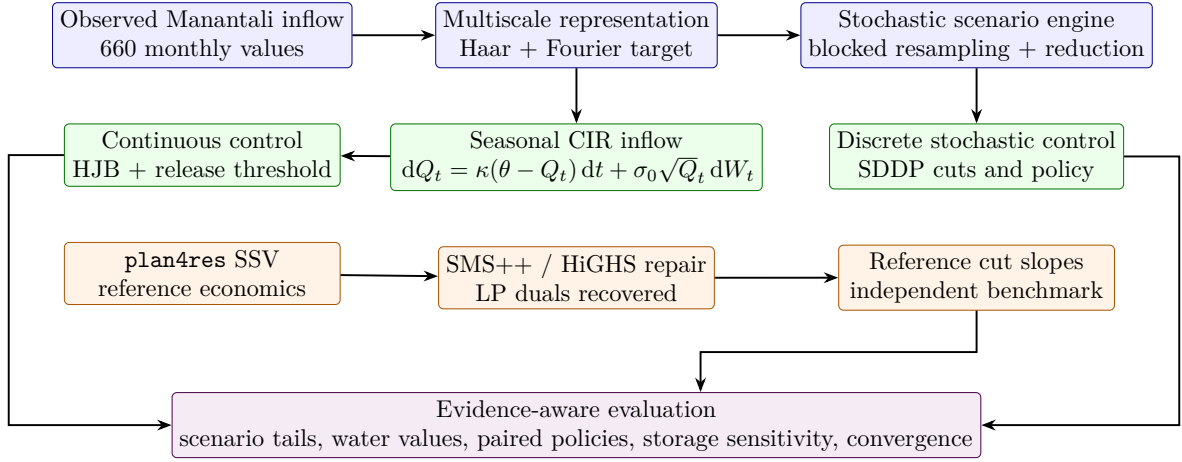
\begin{figure*}[t]
\centering
\resizebox{0.95\textwidth}{!}{%
\begin{tikzpicture}[
node distance=8mm and 12mm,
base/.style={draw,rounded corners=2pt,align=center,font=\small,minimum height=9mm,inner sep=3pt},
hydro/.style={base,fill=blue!7,draw=blue!55!black,minimum width=40mm},
control/.style={base,fill=green!8,draw=green!45!black,minimum width=40mm},
bench/.style={base,fill=orange!10,draw=orange!65!black,minimum width=40mm},
result/.style={base,fill=violet!8,draw=violet!55!black,minimum width=58mm},
arrow/.style={-{Stealth[length=2.1mm]},thick}
]
\node[hydro] (obs) {Observed Manantali inflow\\660 monthly values};
\node[hydro,right=of obs] (season) {Multiscale representation\\Haar + Fourier target};
\node[hydro,right=of season] (scen) {Stochastic scenario engine\\blocked resampling + reduction};

\node[control,below=of obs] (hjb) {Continuous control\\HJB + release threshold};
\node[control,below=of season] (cir) {Seasonal CIR inflow\\$\dif Q_t=\kappa(\theta-Q_t)\dif t+\sigma_0\sqrt Q_t\dif W_t$};
\node[control,below=of scen] (sddp) {Discrete stochastic control\\SDDP cuts and policy};

\node[bench,below=of hjb] (p4r) {\texttt{plan4res} SSV\\reference economics};
\node[bench,below=of cir] (repair) {SMS++ / HiGHS repair\\LP duals recovered};
\node[bench,below=of sddp] (cuts) {Reference cut slopes\\independent benchmark};

\node[result,below=12mm of repair] (eval)
{Evidence-aware evaluation\\scenario tails, water values, paired policies, storage sensitivity, convergence};

\draw[arrow] (obs)--(season);
\draw[arrow] (season)--(scen);
\draw[arrow] (season)--(cir);
\draw[arrow] (cir)--(hjb);
\draw[arrow] (scen)--(sddp);
\draw[arrow] (p4r)--(repair);
\draw[arrow] (repair)--(cuts);
\draw[arrow] (hjb.west) -- ++(-8mm,0) |- (eval.west);
\draw[arrow] (sddp.east) -- ++(8mm,0) |- (eval.east);
\draw[arrow] (cuts.south) -- ++(0,-6mm) -| ([xshift=18mm]eval.north);
\end{tikzpicture}%
}
\caption{Scientific architecture of the study. MOSSHOOS is the central stochastic inflow-to-storage-value pipeline, with a discrete SDDP branch and a continuous CIR--HJB branch. The \texttt{plan4res} lane is an independent economic and software benchmark.}
\label{fig:workflow}
\end{figure*}

\subsection{Seasonal inflow representation and scenario construction}\label{sec:scenario}
Let $x_t$ denote monthly energy inflow. An orthonormal Haar representation writes
\begin{equation}
 x_t=A_J(t)+\sum_{j=1}^{J}D_j(t),
 \label{eq:haar}
\end{equation}
where $A_J$ is a coarse component and $D_j$ are detail components
\cite{mallat2009,daubechies1992}. The archived monthly pipeline uses $J=1$ and
captures 86.517\% of the signal energy in $A_1$. This is a parsimonious
low-frequency/residual split, not a claim that $A_1$ is a pure annual mode. At
monthly resolution it can mix annual structure with slower persistence. The
daily companion analysis resolves the annual cycle principally at a different
wavelet band and retains slower variability separately \cite{affognon2026wavelet};
therefore the daily and monthly decompositions answer different questions.

The seasonal target is represented by a finite Fourier regression,
\begin{equation}
 \theta(t)=a_0+\sum_{h=1}^{H}
 \left[a_h\cos(2\pi h t)+b_h\sin(2\pi h t)\right],
 \label{eq:fourier}
\end{equation}
with coefficients fitted to the coarse monthly component. In the archive this
fit has $R^2=0.84669$. Define residuals $r_t=x_t-\theta(t)$. To retain short-lag
dependence, four-month residual blocks are resampled and added back to the
seasonal target. For candidate scenario $k$ and calendar month $m$,
\begin{equation}
 q_m^{(k)}=\max\{0,\theta_m+r_{m}^{*(k)}\},
 \label{eq:scenario-reconstruct}
\end{equation}
where $r^{*(k)}$ is formed by concatenating sampled blocks. Candidate paths are
then reduced to ten representative annual trajectories with weights $p_k$,
$\sum_kp_k=1$. The reduced law is therefore
\begin{equation}
 \widehat{\mathbb P}_Q=\sum_{k=1}^{10}p_k\,\delta_{q^{(k)}}.
 \label{eq:reduced-law}
\end{equation}
This discrete law drives the SDDP branch. Scenario quality is assessed separately
by a proper multivariate energy score, a marginal KS statistic, autocorrelation
error, and p10/p90 errors \cite{kaut2007,gneiting2007}. A good energy score alone
is not treated as proof of correct tails.

\subsection{Continuous seasonal square-root inflow model}\label{sec:cir}
The continuous cross-check uses
\begin{equation}
 \dif Q_t=\kappa\{\theta(t)-Q_t\}\dif t+\sigma_0\sqrt{Q_t}\dif W_t,
 \qquad Q_0\ge0,
 \label{eq:cir}
\end{equation}
where $\theta(t)$ is the seasonal target in \eqref{eq:fourier}. The square-root
volatility is attractive for nonnegative inflow because fluctuations shrink near
the origin, but positivity must be treated carefully when the target is
seasonal. The archived calibration uses $\kappa=15\,\mathrm{yr}^{-1}$ and
$\sigma_0=21.77$ in the adopted energy-inflow scaling. From the fitted Fourier
target, $\theta_{\min}=4.6639$, $\bar\theta=102.6318$, and
$\theta_{\max}=298.8591$~GWh/month. Consequently
\begin{equation}
 F(t)=\frac{2\kappa\theta(t)}{\sigma_0^2}
 \label{eq:feller-ratio}
\end{equation}
falls to $0.295$ in the driest fitted month, even though its mean-target value is
$6.497$. The uniform Feller condition therefore fails.

\begin{proposition}[First moment and positivity diagnostic]\label{prop:cir}
Assume $\theta:[0,T]\to(0,\infty)$ is continuous and bounded and that
\eqref{eq:cir} is solved with $Q_0\ge0$. Then, whenever the stochastic integral
is integrable,
\begin{equation}
 \E[Q_t]=e^{-\kappa t}Q_0+
 \kappa\int_0^t e^{-\kappa(t-s)}\theta(s)\dif s.
 \label{eq:first-moment}
\end{equation}
Moreover, the uniform condition
$2\kappa\inf_{s\in[0,T]}\theta(s)\ge\sigma_0^2$ is sufficient to keep zero
inaccessible. If that inequality fails, strict positivity cannot be inferred
from the Feller criterion and a nonnegativity-preserving boundary treatment is
required numerically.
\end{proposition}

\noindent A full It\^o-calculus derivation, the exact integrating-factor representation,
first- and second-moment formulas, the constant-target CIR transition law, and the
boundary argument underlying this proposition are collected in
Appendix~\ref{app:ito}.

Figure~\ref{fig:cir} visualizes the fitted seasonal target and this boundary
diagnostic. The point of the figure is not to claim that a Feller failure implies
negative inflow---the CIR state remains a nonnegative square-root process under
appropriate construction---but to show why strict interior positivity is not an
available simplification for the HJB solver.

\begin{figure}[t]
\centering
\includegraphics[width=0.95\textwidth]{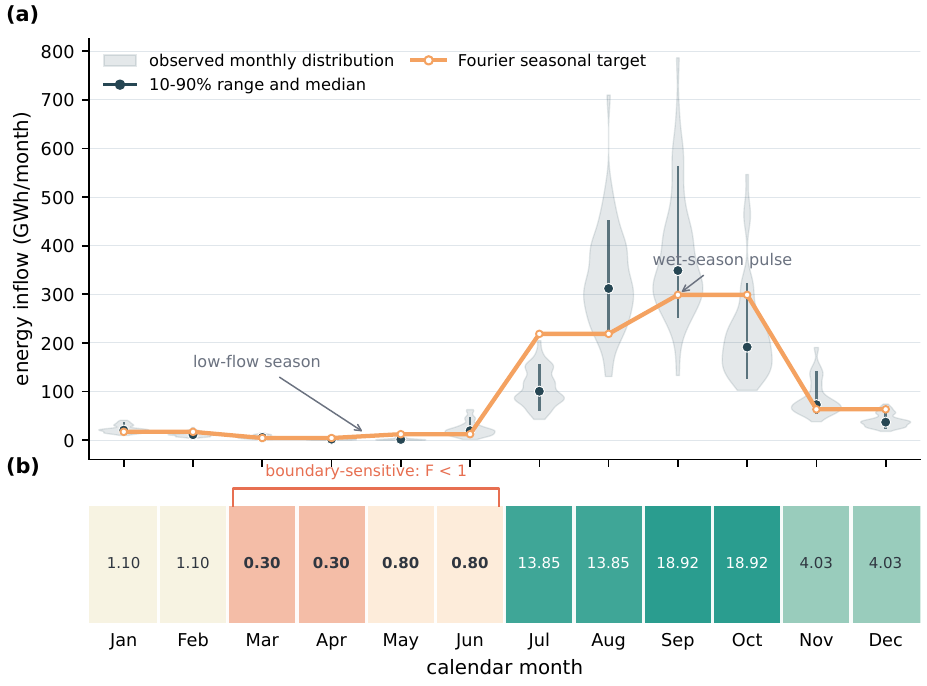}
\caption{Seasonal inflow distribution and boundary diagnostic. (a) Monthly violin envelopes summarize the 55 annual observations; vertical bars show the observed 10--90\% range, dark points the monthly medians, and the orange curve the Fourier seasonal target. (b) Calendar heat strip of the monthly Feller ratio; warm cells with values below one identify months for which the uniform strict-positivity criterion is unavailable. The diagnostic concerns boundary accessibility of the continuous approximation, not physically negative inflow.}
\label{fig:cir}
\end{figure}

\subsection{Discrete stochastic reservoir valuation}\label{sec:sddp}
Let $S_t$ be stored energy, $q_{tk}$ inflow in scenario $k$, $u_t$ turbine
release, $w_t$ spill, $g_t$ thermal backstop generation, and $z_t$ unserved
energy. A risk-neutral stage problem is
\begin{align}
Q_t(S_t)=\min_{u_t,w_t,g_t,z_t}\quad &
 c_g g_t+c_z z_t+c_w w_t
 +\gamma\sum_k p_k Q_{t+1}(S_{t+1,k})
 \label{eq:sddpobj}\\
\text{s.t.}\quad &S_{t+1,k}=S_t+q_{tk}-u_t-w_t,
 \label{eq:balance}\\
&0\le S_{t+1,k}\le S_{\max},\qquad 0\le u_t\le U_{\max},\nonumber\\
&\eta u_t+g_t+z_t=d_t,\qquad g_t,z_t,w_t\ge0.\nonumber
\end{align}
The continuation function is approximated from below by affine cuts,
\begin{equation}
 \widehat Q_{t+1}(S)=\max_{\ell\in\mathcal C_{t+1}}
 \{a_{t\ell}+b_{t\ell}S\}.
 \label{eq:cuts}
\end{equation}

\begin{proposition}[Active SDDP slope as water value]\label{prop:sddp}
Suppose the stage problem is convex and an active continuation cut
$a_{t\ell^*}+b_{t\ell^*}S$ is unique at the evaluated storage state. Then the
local marginal value of an additional unit of stored energy is
\begin{equation}
 \lambda_t^W=-b_{t\ell^*}.
 \label{eq:sddp-water}
\end{equation}
Equivalently, $-b_{t\ell^*}$ is the shadow reduction in optimal future cost
associated with increasing storage.
\end{proposition}

\noindent The complete parametric-LP and dual-sensitivity derivation is given in
Appendix~\ref{app:sddp}. With the storage balance written so that its dual
multiplier is $\pi$, the cut slope is the negative probability-weighted dual,
and the positive water value is the corresponding probability-weighted shadow
price.

\subsection{Continuous-time HJB characterization and release threshold}\label{sec:hjb}
The storage state follows
\begin{equation}
 \dif S_t=(Q_t-u_t-w_t)\dif t,\qquad 0\le S_t\le S_{\max},
 \label{eq:storage-sde}
\end{equation}
and the finite-horizon objective is
\begin{multline}
 V(t,s,q)=\inf_{u,w}\E_{t,s,q}\Bigg[
 \int_t^T e^{-\rho(\tau-t)}
 \ell(\tau,S_\tau,Q_\tau,u_\tau,w_\tau)\dif\tau +e^{-\rho(T-t)}g(S_T)\Bigg].
 \label{eq:objective}
\end{multline}
Dynamic programming yields, in the state-constrained viscosity sense,
\begin{multline}
 -V_t=\min_{(u,w)\in\mathcal U(s,q)}\Big\{
 \ell+(q-u-w)V_s+\kappa[\theta(t)-q]V_q +\tfrac12\sigma_0^2 qV_{qq}-\rho V\Big\},
 \qquad V(T,s,q)=g(s).
 \label{eq:hjb}
\end{multline}
The state constraints enter through the admissible control set: release is
limited by water availability near $s=0$, while spill becomes active at the upper
storage boundary.

\begin{proposition}[Marginal release threshold]\label{prop:threshold}
Assume $u\mapsto\ell(t,s,q,u,w)$ is convex and differentiable for fixed
$(t,s,q,w)$, and define the marginal avoided system cost
$m(t,s,q,u)=-\partial\ell/\partial u$ and the continuous water value
$\lambda_W=-V_s$. If $m$ is nonincreasing in $u$, the HJB minimizer satisfies
\begin{equation}
 u^*=\begin{cases}
 0, & m(t,s,q,0)\le\lambda_W,\\
 U_{\max}, & m(t,s,q,U_{\max})\ge\lambda_W,\\
 u:\ m(t,s,q,u)=\lambda_W, & \text{otherwise},
 \end{cases}
 \label{eq:threshold}
\end{equation}
subject to storage feasibility.
\end{proposition}

\noindent Appendix~\ref{app:hjb} derives \eqref{eq:hjb} directly from the dynamic
programming principle and It\^o's formula, states a verification result, and gives
the full KKT proof of \eqref{eq:threshold}, including an explicit quadratic-cost
specialization.

Propositions~\ref{prop:sddp} and \ref{prop:threshold} clarify the conceptual
bridge between the two control branches. They do not imply that a monthly SDDP
sequence and a continuous HJB trajectory must coincide pointwise; they identify
the common marginal economic object that can be benchmarked after clocks,
inflows, costs, and terminal conditions are harmonized.

\subsection{Haar approximation rate used in the numerical diagnostic}
The archived HJB calculation includes a Haar storage-basis convergence table.
The observed slope should not be called a universal convergence theorem for the
entire nonlinear HJB solver. What can be justified directly is the approximation
rate of the storage projection.

\begin{theorem}[Haar projection error for an $H^1$ storage profile]\label{thm:haar}
Let $f\in H^1(0,S_{\max})$ and let $P_Jf$ be its piecewise-constant Haar
projection on $2^J$ equal storage cells. With $M=2^J$ basis cells,
\begin{equation}
 \|f-P_Jf\|_{L^2(0,S_{\max})}
 \le C\,S_{\max}M^{-1}\|f'\|_{L^2(0,S_{\max})},
 \label{eq:haar-rate}
\end{equation}
for a constant $C$ independent of $M$. Hence the storage-projection error is
$O(M^{-1})$ for $H^1$ profiles.
\end{theorem}

\noindent A complete cellwise proof and the distinction between projection error and
full HJB discretization error are given in Appendix~\ref{app:haar}.

\subsection{Independent \texttt{plan4res} benchmark and solver repair}\label{sec:p4r}
The \texttt{plan4res} benchmark is retained because it provides an independent
large-system valuation of the same marginal storage object. Its SSV Bellman
recursion can be written schematically as
\begin{align}
V_n(E_n,\xi_n)&=\min_{x_n\in\mathcal X_n}
\left\{C_n(x_n;\xi_n)+\E_n[V_{n+1}(E_{n+1},\xi_{n+1})]\right\},
\label{eq:p4r-bellman}\\
E_{n+1}&=E_n+A_n-H_n-W_n,
\qquad 0\le E_{n+1}\le E_{\max}.
\label{eq:p4r-balance}
\end{align}
Its continuation function is represented by Benders cuts analogous to
\eqref{eq:cuts}. In the tested SMS++ build, an integer overload of
\texttt{set\_par} hides the inherited double-valued overload, so the requested
integrality-relaxation parameter is assigned in the wrong parameter space. HiGHS
then receives a MIP rather than the LP required for dual-based cut generation.
Restoring the base overload set,
\begin{center}
\texttt{using CDASolver::set\_par;}
\end{center}
causes the relaxation flag to be activated, the subproblem to be solved as an LP,
and the required storage-balance dual and Benders cuts to be returned
\cite{huangfu2018}. The downstream SSV, SIM, and CEM entry points then complete
in the tested configuration. This establishes local open-source reproducibility;
it is not a claim that the source patch has already been accepted upstream.

\begin{algorithmblock}{Stochastic valuation and independent benchmark}
\begin{enumerate}[leftmargin=1.5em,itemsep=1pt]
\item Fit the monthly coarse/seasonal representation and construct blocked candidate inflow paths.
\item Reduce the candidate paths to weighted scenarios and quantify marginal, dependence, and tail errors.
\item Solve the discrete MOSSHOOS SDDP problem and export active cut slopes and policies.
\item Solve the continuous CIR--HJB benchmark with a nonnegative inflow-boundary treatment and compute $-V_s$.
\item Repair and execute the \texttt{plan4res} SSV LP path; export its active cut slopes.
\item Compare marginal water-value levels without forcing a false stage-to-month alignment.
\item Evaluate stochastic and deterministic-equivalent policies on paired Monte Carlo paths and report uncertainty intervals.
\end{enumerate}
\end{algorithmblock}

\subsection{Evaluation and uncertainty quantification}\label{sec:eval}
The 50 stochastic-policy and deterministic-equivalent rows share run and scenario
indices. For outcome $y$, define
$\Delta_i=y_i^{\MS}-y_i^{\MD}$. We resample the 50 paired differences 20,000
times with a fixed seed and use the 2.5th and 97.5th percentiles as the reported
95\% bootstrap interval \cite{efron1993}. An interval containing zero is reported
as ``no detectable difference'' rather than ``equality.''

For the merit-order comparison, a \texttt{plan4res} stage is classified as
matching a listed thermal/scarcity level $c_g$ when
\begin{equation}
 \min_g|\lambda_n-c_g|\le\varepsilon_{\rm match},
 \qquad \varepsilon_{\rm match}=10^{-3}\ \mathrm{USD/MWh}.
 \label{eq:tolerance}
\end{equation}
This tolerance reproduces the archived 12-of-18 classification; the count is
unchanged at $10^{-2}$ but decreases at $10^{-4}$~USD/MWh, so the tolerance is
reported explicitly.

\section{Results}\label{sec:results}

\subsection{Seasonal inflow representation and scenario adequacy}
Figure~\ref{fig:inflow} overlays the observed monthly distribution, Fourier
seasonal target, ten reduced scenarios, and their weighted mean in one view. The
scenario set preserves the wet-season timing but places too much mass above the
observed median during the wet season and does not reproduce the lower tail
satisfactorily.

\begin{figure}[t]
\centering
\includegraphics[width=0.95\textwidth]{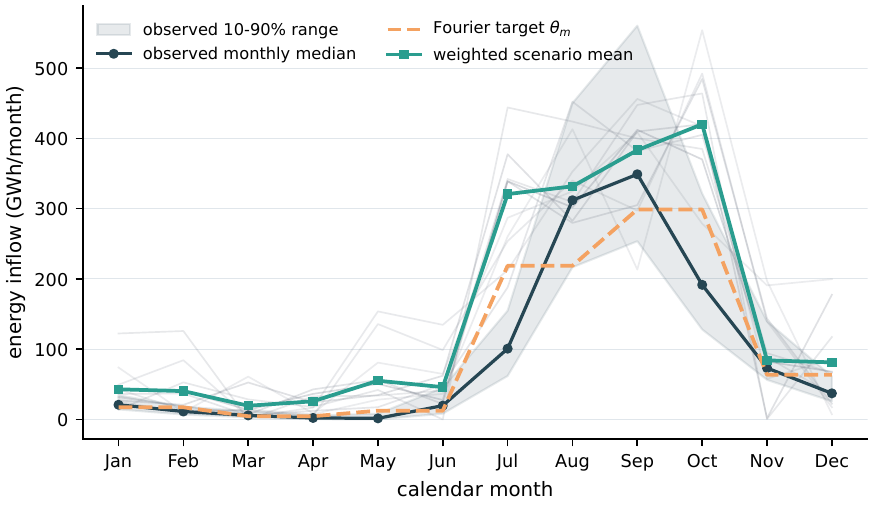}
\caption{Observed Manantali seasonality and the reduced stochastic scenario set. The single figure replaces the earlier two-panel presentation so the reader can see directly how the reduced paths relate to the observed envelope and seasonal target.}
\label{fig:inflow}
\end{figure}

The blocked seasonal construction has energy score 291.01 versus 555.60 for
naive independent monthly sampling, a 47.62\% improvement. However, the marginal
KS statistic is 0.927 ($p=2.18\times10^{-56}$), maximum autocorrelation error is
0.168, p10 relative error is 103.86\%, and p90 relative error is 18.47\%.
Table~\ref{tab:scenario_diag} makes the interpretation explicit.

\begin{table}[t]
\centering
\caption{Scenario diagnostics and their scientific interpretation.}
\label{tab:scenario_diag}
\begin{tabular}{>{\raggedright\arraybackslash}p{0.27\textwidth}p{0.22\textwidth}>{\raggedright\arraybackslash}p{0.42\textwidth}}
\toprule
Diagnostic & Archived value & Interpretation \\
\midrule
Energy score & 291.01 vs 555.60 & Joint forecast geometry improves strongly relative to independent monthly sampling. \\
Marginal KS & 0.927; $p=2.18\times10^{-56}$ & Marginal scenario distribution is not calibrated to the observed sample. \\
Maximum ACF error & 0.168 & Short-lag dependence is only partially reproduced. \\
p10 relative error & 103.86\% & Dry-tail behavior is inadequate for operational drought-risk inference. \\
p90 relative error & 18.47\% & Upper-tail representation is better than p10 but still imperfect. \\
\bottomrule
\end{tabular}
\end{table}

The main stochastic-model limitation is therefore hydrological, not algorithmic:
blocking improves one joint score but does not yet capture the low-flow tail that
should create the largest value for hedging.

\subsection{Continuous inflow model and boundary behavior}
The archived CIR parameters imply a short mean-reversion timescale relative to the
annual seasonal cycle, while the fitted seasonal target varies strongly between
dry and wet months. Figure~\ref{fig:cir} shows that the resulting Feller ratio is
well above one in the wet season but falls below one around the lowest seasonal
target. This corrects the earlier global-positivity interpretation: the data do
not support using a single mean-based Feller ratio to exclude the zero boundary.
The HJB implementation should therefore employ a nonnegative discretization or
explicit degenerate-boundary condition near $q=0$.

\subsection{Water-value benchmark: stochastic control first, software benchmark second}
Figure~\ref{fig:water} compares the marginal-value distributions without
pretending that the 18 \texttt{plan4res} stages and 12 MOSSHOOS months share a
common clock. Panel (a) isolates the operating-value range and displays the two
sets as strip distributions with their means. Panel (b) restores the full
\texttt{plan4res} scarcity spectrum on a logarithmic scale and overlays the narrow
MOSSHOOS range, making clear both the agreement near the dominant thermal level
and the much larger scarcity values present in the benchmark.

\begin{figure}[t]
\centering
\includegraphics[width=0.95\textwidth]{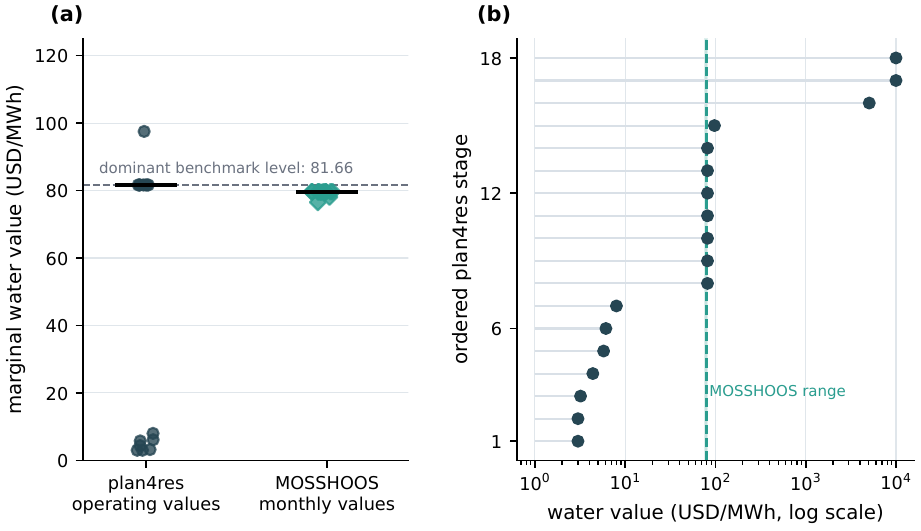}
\caption{Marginal water-value benchmark without artificial temporal alignment. (a) Operating-range values for \texttt{plan4res} and MOSSHOOS, with horizontal median markers. (b) Ordered \texttt{plan4res} values on a logarithmic scale; the shaded vertical band is the MOSSHOOS range. The dominant \texttt{plan4res} level is 81.6649~USD/MWh and the MOSSHOOS mean is 79.2466~USD/MWh.}
\label{fig:water}
\end{figure}

The mean-level agreement is 97.04\%, corresponding to a 2.96\% difference.
Under the declared matching tolerance, 12 of 18 \texttt{plan4res} cut slopes
coincide with listed thermal/scarcity levels and the remaining slopes lie at
intermediate probability-weighted levels. The comparison therefore supports
economic consistency of the marginal object, not equality of trajectories.

\subsection{Paired policy comparison}
The raw paired tables give mean annual costs of 334,746.26~USD/cycle for the
stochastic policy and 334,738.66~USD/cycle for the deterministic-equivalent
policy. The paired mean difference is $+7.60$~USD/cycle with 95\% bootstrap
interval $[-57.64,55.27]$. Mean deficits are 181.3541 and 181.3483~GWh,
respectively, for a paired difference of $+0.0059$~GWh with interval
$[-0.0639,0.0678]$~GWh. Figure~\ref{fig:effects} shows the run-level paired differences together
with the paired-bootstrap mean and interval in the original physical units, so
sampling variability and the estimated effect can be read from the same panel.

\begin{figure}[t]
\centering
\includegraphics[width=0.94\textwidth]{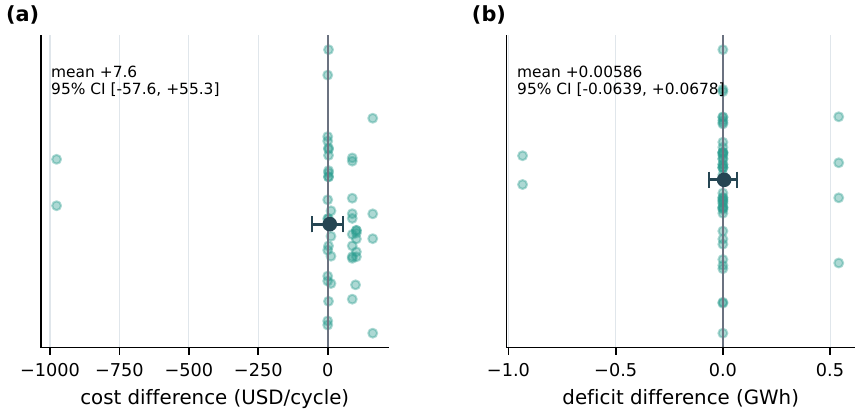}
\caption{Paired stochastic-policy effects. (a) Annual operating-cost differences and (b) annual deficit differences, defined as stochastic minus deterministic-equivalent policy. Points are run-level paired differences; the larger marker and horizontal bar give the paired mean and 95\% bootstrap interval. Both intervals contain zero.}
\label{fig:effects}
\end{figure}

Thus no stochastic-policy advantage is detected in this archived operating
configuration. This does not invalidate stochastic optimization; it is consistent
with either an operating regime where the deterministic equivalent is already
adequate or a scenario law that does not yet reproduce the rare/persistent events
for which hedging should matter.

\subsection{Storage-regime sensitivity}
The storage sweep indicates that water-value control improves on greedy release
by 1.53\% at 0.1 times mean-monthly inflow of usable storage and by 17.11\% at
four times mean inflow. The incremental stochastic gain over the deterministic
equivalent remains approximately zero over the principal archived operating range.
At six times mean-monthly inflow, the reported 9.58\% stochastic gain is retained
as a storage-sensitivity result from the archived project summary and is therefore
interpreted cautiously. Figure~\ref{fig:sweep} summarizes both performance contrasts
across storage ratios and distinguishes the six-times sensitivity point from the
principal operating-range results.

\begin{figure}[t]
\centering
\includegraphics[width=0.95\textwidth]{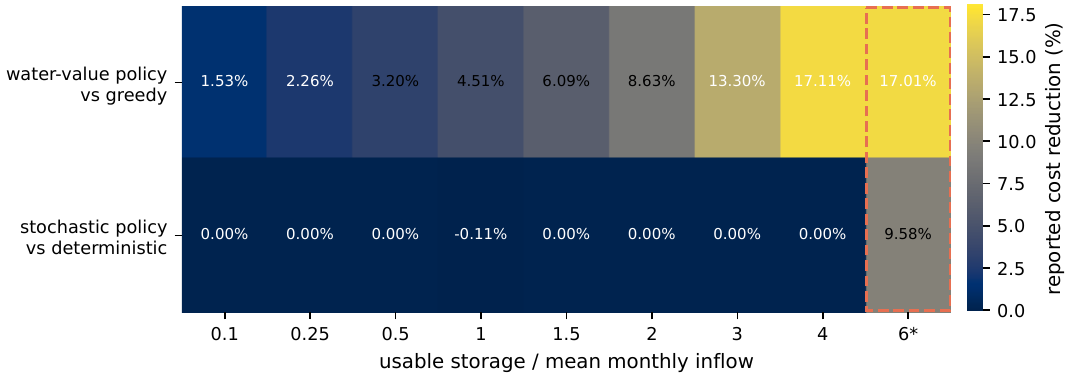}
\caption{Storage-regime sensitivity shown as a heat map of reported cost reductions. The first row compares the water-value policy with greedy release; the second compares the stochastic and deterministic-equivalent policies. The $6\times$ column (asterisk and dashed outline) denotes the additional high-storage sensitivity result reported in the archived project summary.}
\label{fig:sweep}
\end{figure}

\subsection{Storage-basis convergence}
The five archived wavelet--Galerkin points yield a global log--log slope of
$-0.9928$ (Fig.~\ref{fig:conv}). With only five ordered refinement levels, a
histogram would discard the convergence structure rather than clarify it. We
therefore show both the error curve and the local empirical orders between
successive basis sizes. The global behavior is consistent with the $O(M^{-1})$
Haar storage-projection rate in Theorem~\ref{thm:haar}, while the local orders show
the expected finite-resolution variation around one. This is not, on its own, a
proof that the entire state-constrained nonlinear HJB scheme has first-order error
in every state and time variable.

\begin{figure}[t]
\centering
\includegraphics[width=0.95\textwidth]{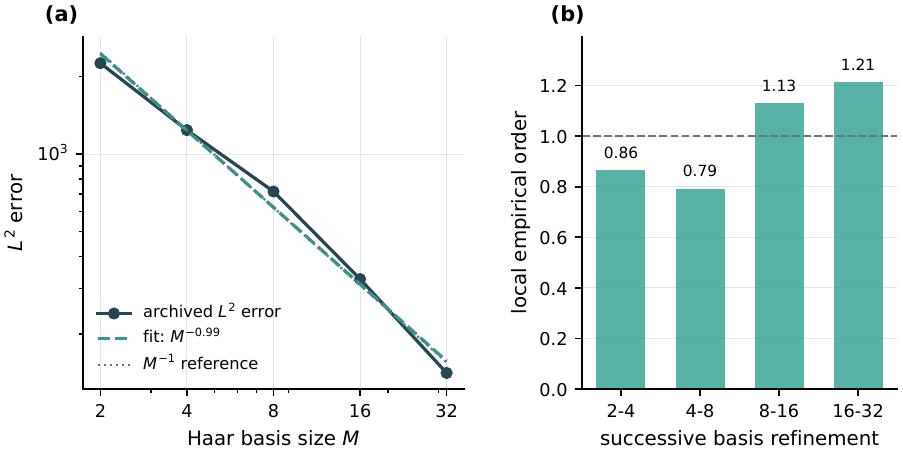}
\caption{Storage-basis convergence diagnostic. (a) Archived $L^2$ errors, global log--log fit, and an $M^{-1}$ reference slope. (b) Local empirical convergence orders between successive basis refinements. The global fitted slope is $-0.9928$, consistent with first-order Haar projection of an $H^1$ storage profile.}
\label{fig:conv}
\end{figure}

\section{Discussion}\label{sec:discussion}
The results support three linked conclusions. First, the scientific contribution
is the propagation of hydrological uncertainty into seasonal storage value through
two complementary control representations: a weighted scenario law for SDDP and
a continuous seasonal square-root diffusion for HJB analysis. Second, the close
agreement in the economically active water-value range is useful cross-validation,
but it does not compensate for the weak lower-tail calibration of the current
scenario generator. Third, the absence of a detectable stochastic-policy advantage
at the archived operating point should be interpreted together with that tail
limitation rather than as evidence that hydrological uncertainty is unimportant.
The discussion therefore proceeds from interpretation of the water-value results,
through hydrological limitations, to the role of the independent benchmark and the
system-level extensions required for operational use.

\subsection{Interpretation of stochastic water-value results}
The MOSSHOOS framework links uncertain inflow to a marginal storage opportunity
cost. In the discrete formulation, the active SDDP cut slope is the shadow price of
an additional unit of stored energy. In the continuous formulation, the same object
appears as the storage derivative of the HJB value function. This shared economic
interpretation explains why the two control representations can be compared even
though their numerical machinery and temporal discretizations differ.

The paired Monte Carlo experiment shows that the stochastic and
deterministic-equivalent policies are statistically indistinguishable at the archived
operating point. That result is informative but should not be generalized beyond the
tested configuration. The scenario diagnostics simultaneously show a severe lower-tail
deficiency, so the current stochastic law may not generate the persistent drought
states in which precautionary storage has the greatest value. The appropriate
interpretation is therefore conditional: no incremental stochastic-policy benefit is
detected under the archived scenario law and operating regime, while the potential
value of improved hydrological uncertainty representation remains an open empirical
question.

\subsection{Multiscale hydrology and tail-risk limitations}
The single-level monthly Haar split is best viewed as a transparent archival baseline.
It separates a coarse low-frequency component before blocked residual resampling, but
it cannot isolate the annual and slower interannual scales that may govern persistent
drought. The daily companion analysis \cite{affognon2026wavelet} shows that annual and
multi-year components can be separated explicitly. A direct multiscale ablation would
therefore be the most informative next test: rebuild the monthly scenario law with
deeper or cross-scale features and compare energy score, p10/p90 errors, dry-spell
duration, autocorrelation, and out-of-sample operating cost under an identical policy
evaluation design.

The time-varying Feller ratio provides a second warning about the dry-season regime.
Because the ratio falls below one around the minimum seasonal target, the continuous
square-root approximation cannot rely on a uniform strict-positivity guarantee. This
does not invalidate the diffusion model, but it requires an explicit nonnegative
numerical treatment at the inflow boundary and cautions against calibrating the model
from average conditions alone.

A mathematically natural tail extension, if event diagnostics justify discontinuous
innovations, is the state-constrained jump diffusion
\begin{equation}
 \dif Q_t=\kappa\{\theta(t)-Q_{t-}\}\dif t
 +\sigma_0\sqrt{Q_{t-}}\dif W_t
 +\int_{\mathcal Z}\xi(Q_{t-},z)\,\widetilde N(\dif t,\dif z),
 \qquad Q_{t-}+\xi(Q_{t-},z)\ge0,
 \label{eq:jump-extension}
\end{equation}
where $\widetilde N$ is a compensated Poisson random measure with L\'evy measure
$\nu$. The HJB operator would then gain the nonlocal term
\begin{equation}
 \mathcal J V(t,s,q)=\int_{\mathcal Z}\!\left[
 V(t,s,q+\xi)-V(t,s,q)-\xi V_q(t,s,q)\mathbf 1_{\{|\xi|<1\}}
 \right]\nu(\dif z).
 \label{eq:jump-generator}
\end{equation}
Equations~\eqref{eq:jump-extension}--\eqref{eq:jump-generator} are a candidate
extension only, not a fitted result of the present study. Hydrological event analysis
must first justify jump occurrence and sign, while $\xi$ and $\nu$ must be estimated
and validated out of sample; the nonnegativity constraint prevents a jump law from
creating physically meaningless negative inflow \cite{applebaum2009,katz2002}.

\subsection{Independent plan4res benchmark and reproducibility}
The \texttt{plan4res} comparison is valuable because it supplies an independent
economic benchmark for the marginal value of stored water. In the tested open-source
configuration, the solver-interface defect prevented recovery of the LP duals needed
for Benders cuts. Restoring the correct LP formulation makes those duals available and
allows the SSV chain to execute with HiGHS. The resulting dominant marginal level of
81.6649~USD/MWh is close to the MOSSHOOS mean of 79.2466~USD/MWh, supporting
consistency of the economic scale.

This agreement should nevertheless be interpreted at the level of the marginal object,
not as equality of complete trajectories. The 18 \texttt{plan4res} stages and the 12
MOSSHOOS months do not share a common time index, and the benchmark contains scarcity
levels far outside the normal MOSSHOOS operating range. The contribution of the
open-source repair is therefore methodological and reproducibility-oriented: it makes
an independent reference executable and exposes the dual quantities required for a
transparent water-value comparison, while the central scientific claim remains the
stochastic inflow-to-storage-value relationship.

\subsection{System-level implications under variable renewables and ancillary services}
The present benchmark values stored water against an exogenous reference generation
fleet. A more operational Senegalese application would construct joint
inflow--solar--wind--load scenarios and value hydropower against stochastic residual
load. The relevant outputs would then include not only marginal water value but also
renewable curtailment, thermal generation and starts, deficits, and reserve scarcity.
This extension is important because the system value of seasonal storage may increase
as variable renewable penetration changes the timing and frequency of scarcity.

The current objective also captures energy arbitrage but not spinning or regulating
reserve, fast load following, frequency response, or voltage/reactive-power support.
Those services require finer chronological resolution and, for voltage support, an
explicit network representation. They should be introduced through co-optimization
rather than by attaching an external premium to the monthly water value. A broader
triangulation with GENeSYS-MOD and/or openTEPES \cite{loeffler2017,opentepes2022}
would be meaningful only after harmonizing inflow, residual load, unit availability,
terminal conditions, discounting, and scarcity penalties; comparison should then use
common system outputs rather than native stage indices.

\subsection{Physical scope and economic interpretation for the OMVS system}
Operational conclusions ultimately require a coupled representation of the wider OMVS
hydraulic system. Manantali, Gouina, and F\'elou should be linked through mass balance,
routing and travel time, head-dependent conversion, turbine limits, and environmental
and multi-purpose releases; Diama should enter through its downstream regulation and
water-management role rather than be treated as an identical storage plant
\cite{loucks2017}. The conceptual cascade cross-check in the InFlow preprint
\cite{affognon2026inflow} is methodological evidence, not an OMVS calibration.

For local economic interpretation, using an illustrative exchange rate of
570~XOF per USD, the two benchmark values correspond to approximately 46.55 and
45.17~XOF/kWh, respectively. These quantities are marginal storage opportunity
costs, not average tariffs or observed generation costs. Any policy comparison with
SENELEC thermal production should therefore use a verified marginal avoided cost
for the relevant diesel- or HFO-fired units and operating period.

\section{Conclusion}\label{sec:conclusion}
This study develops and evaluates a stochastic seasonal-storage valuation
framework for Manantali, with an independent open-source benchmark. The MOSSHOOS
workflow begins with uncertain inflow, constructs a weighted seasonal scenario law
for SDDP, introduces a continuous seasonal square-root diffusion for HJB analysis,
and connects both control representations through the marginal value of stored water. The revised
mathematical statements make that connection explicit: active SDDP slopes are
storage shadow prices, HJB release follows a marginal-cost threshold, and the
observed Haar storage-basis convergence is consistent with the expected $H^1$
projection rate.

The empirical evidence is deliberately mixed. The reduced scenarios improve the
energy score by 47.6\%, and MOSSHOOS produces a mean marginal value within 3\%
of the independently recovered \texttt{plan4res} benchmark. At the same time,
the lower-tail calibration is poor, the time-varying CIR Feller ratio falls below
one in the driest season, and the paired Monte Carlo experiment finds no detectable
stochastic-policy advantage at the archived operating configuration. These are
not contradictory results: they identify hydrological tail representation as the
current bottleneck.

Future work should prioritize deeper multiscale scenario construction, persistent
drought and spill-tail validation, common residual-load experiments with solar and
wind, independent replication of the high-storage regime, and eventual calibration
of the coupled OMVS system.
The repaired \texttt{plan4res} chain remains valuable as an independent reference,
but the scientific success criterion is whether improved stochastic inflow models
produce stable, interpretable, and out-of-sample useful water values.

\section*{Acknowledgments}
This research benefited from the support of the Fondation Math\'ematique Jacques
Hadamard (FMJH) through the Programme Gaspard Monge pour l'Optimisation, la
recherche op\'erationnelle et leurs interactions avec les sciences des donn\'ees
(PGMO). The authors acknowledge the \texttt{plan4res} consortium for the open
modeling framework; \texttt{plan4res} received funding from the European Union's
Horizon 2020 research and innovation programme under grant agreement No.~773897.
The authors also acknowledge the MOSSHOOS project partners and the institutions
responsible for the Senegal River inflow records.

\section*{Data and Code Availability}
The machine-readable data tables supporting the analyses and figures, together with
the plotting scripts and manuscript source files, are provided in the accompanying
reproducibility package.

\appendix
\section{It\^o calculus and exact identities for the seasonal inflow diffusion}
\label{app:ito}
This appendix develops the stochastic-calculus results used in
Section~\ref{sec:cir}. Let
\begin{equation}
 \dif Q_t=\kappa\{\theta(t)-Q_t\}\dif t+\sigma_0\sqrt{Q_t}\dif W_t,
 \qquad Q_{t_0}=q\ge0,
 \label{eq:app-cir}
\end{equation}
where $\theta$ is deterministic, continuous and bounded on a finite horizon.
The model is time-inhomogeneous because the seasonal target changes with the
calendar. Standard square-root diffusion theory gives a nonnegative weak solution;
strict interior positivity requires a stronger boundary condition than mere
nonnegativity \cite{cir1985,feller1951}.

\subsection{It\^o formula and infinitesimal generator}
For any $f\in C^{1,2}([0,T]\times(0,\infty))$, It\^o's formula applied before a
possible boundary contact gives
\begin{align}
 \dif f(t,Q_t)=&\Big[f_t+\kappa\{\theta(t)-Q_t\}f_q
 +\tfrac12\sigma_0^2Q_t f_{qq}\Big]\dif t \\
 &+\sigma_0\sqrt{Q_t}\,f_q\dif W_t.
 \label{eq:app-ito-general}
\end{align}
Hence the time-dependent infinitesimal generator of the inflow process is
\begin{equation}
 (\mathcal L_t^Q f)(q)=\kappa\{\theta(t)-q\}f_q(q)
 +\tfrac12\sigma_0^2q f_{qq}(q).
 \label{eq:app-generator}
\end{equation}
This is precisely the inflow operator appearing in the HJB equation.

\subsection{Exact integrating-factor representation}
Take $f(t,q)=e^{\kappa t}q$ in \eqref{eq:app-ito-general}. Because
$f_t=\kappa e^{\kappa t}q$, $f_q=e^{\kappa t}$ and $f_{qq}=0$,
\begin{equation}
 \dif\big(e^{\kappa t}Q_t\big)
 =\kappa e^{\kappa t}\theta(t)\dif t
 +\sigma_0e^{\kappa t}\sqrt{Q_t}\dif W_t.
 \label{eq:app-integrating-factor}
\end{equation}
Integrating from $t_0$ to $t$ and multiplying by $e^{-\kappa t}$ yields the exact
variation-of-constants representation
\begin{align}
 Q_t={}&e^{-\kappa(t-t_0)}q
 +\kappa\int_{t_0}^{t}e^{-\kappa(t-s)}\theta(s)\dif s \\
 &+\sigma_0\int_{t_0}^{t}e^{-\kappa(t-s)}\sqrt{Q_s}\dif W_s.
 \label{eq:app-mild-solution}
\end{align}
Equation~\eqref{eq:app-mild-solution} is an exact solution representation, but it
is not an elementary closed-form sample path because $Q_s$ remains inside the
stochastic integral. For a genuinely time-varying $\theta(t)$ there is therefore
no simple noncentral-$\chi^2$ transition formula of the standard constant-target
CIR form. The representation is nevertheless sufficient for exact moment identities
and for the control generator used below.

\begin{proposition}[First and second moments]\label{prop:app-moments}
Let $m_1(t)=\E[Q_t]$ and $m_2(t)=\E[Q_t^2]$. Under the usual square-integrability
conditions,
\begin{align}
 m_1(t)={}&e^{-\kappa(t-t_0)}q
 +\kappa\int_{t_0}^{t}e^{-\kappa(t-s)}\theta(s)\dif s,
 \label{eq:app-m1}\\
 m_2(t)={}&e^{-2\kappa(t-t_0)}q^2 \\
 &+\int_{t_0}^{t}e^{-2\kappa(t-s)}
 \big(2\kappa\theta(s)+\sigma_0^2\big)m_1(s)\dif s.
 \label{eq:app-m2}
\end{align}
Consequently $\operatorname{Var}(Q_t)=m_2(t)-m_1(t)^2$.
\end{proposition}

\begin{proof}
Taking expectations in \eqref{eq:app-mild-solution} removes the martingale term and
gives \eqref{eq:app-m1}. For the second moment, apply It\^o's formula to
$f(q)=q^2$:
\begin{equation}
 \dif Q_t^2=\Big[(2\kappa\theta(t)+\sigma_0^2)Q_t
 -2\kappa Q_t^2\Big]\dif t
 +2\sigma_0Q_t^{3/2}\dif W_t.
 \label{eq:app-q2}
\end{equation}
Taking expectations produces the linear equation
$m_2'+2\kappa m_2=(2\kappa\theta+\sigma_0^2)m_1$. Multiplication by the
integrating factor $e^{2\kappa t}$ gives \eqref{eq:app-m2}.
\end{proof}

\subsection{Constant-target closed form and the seasonal distinction}
If $\theta(t)\equiv\bar\theta$ is constant and $\Delta=t-t_0$, then the first
two conditional moments reduce to
\begin{align}
 \E[Q_t\mid Q_{t_0}=q]
 &=\bar\theta+(q-\bar\theta)e^{-\kappa\Delta},
 \label{eq:app-cir-mean-constant}\\
 \operatorname{Var}(Q_t\mid q)
 &=\frac{q\sigma_0^2}{\kappa}e^{-\kappa\Delta}
   (1-e^{-\kappa\Delta})
 +\frac{\bar\theta\sigma_0^2}{2\kappa}
   (1-e^{-\kappa\Delta})^2.
 \label{eq:app-cir-var-constant}
\end{align}
Moreover,
\begin{equation}
 Q_t\mid Q_{t_0}=q\ \overset{d}{=}\
 c_\Delta\,\chi'^2_{\nu}(\lambda_\Delta),
 \label{eq:app-cir-transition}
\end{equation}
where
\begin{equation}
 c_\Delta=\frac{\sigma_0^2(1-e^{-\kappa\Delta})}{4\kappa},\qquad
 \nu=\frac{4\kappa\bar\theta}{\sigma_0^2},\qquad
 \lambda_\Delta=\frac{4\kappa e^{-\kappa\Delta}q}
 {\sigma_0^2(1-e^{-\kappa\Delta})}.
 \label{eq:app-cir-transition-par}
\end{equation}
Thus the constant-target model admits exact transition sampling. The seasonal model
used in this article deliberately retains $\theta(t)$, so equations
\eqref{eq:app-mild-solution}--\eqref{eq:app-m2}, rather than
\eqref{eq:app-cir-transition}, are the exact identities that remain valid without
freezing seasonality.

\subsection{Boundary diagnostic from It\^o's formula}
For the stopped process before $Q_t$ reaches a small level $\varepsilon>0$, apply
It\^o's formula to $\log Q_t$:
\begin{equation}
 \dif\log Q_t=\left[
 \frac{\kappa\theta(t)-\tfrac12\sigma_0^2}{Q_t}-\kappa
 \right]\dif t
 +\frac{\sigma_0}{\sqrt{Q_t}}\dif W_t.
 \label{eq:app-logito}
\end{equation}
The singular drift near zero changes sign according to
$2\kappa\theta(t)-\sigma_0^2$. For a constant target, the classical Feller
classification therefore makes zero inaccessible when
$2\kappa\bar\theta\ge\sigma_0^2$. For a seasonal target, the uniform condition
\begin{equation}
 2\kappa\inf_{t\in[0,T]}\theta(t)\ge\sigma_0^2
 \label{eq:app-uniform-feller}
\end{equation}
is a convenient sufficient criterion for strict interior positivity. The fitted
Manantali target violates this uniform inequality in the dry-season minimum, which
is why the numerical HJB treatment retains the $q=0$ boundary instead of assuming
it away. This observation does not imply negative inflow: at $q=0$ the diffusion
coefficient vanishes and the drift $\kappa\theta(t)$ points into the nonnegative
state space.

\section{Dynamic programming, HJB derivation, and optimal release}
\label{app:hjb}
This appendix supplies the control derivation behind
Section~\ref{sec:hjb}. Let $\mathcal F_t$ be the filtration generated by the inflow
noise and let an admissible control $a_t=(u_t,w_t)$ be progressively measurable,
with $0\le u_t\le U_{\max}$, $w_t\ge0$, and with the induced storage path satisfying
$0\le S_t\le S_{\max}$. At $S_t=0$ the viability condition requires
$Q_t-u_t-w_t\ge0$; at $S_t=S_{\max}$ it requires $Q_t-u_t-w_t\le0$ whenever the
boundary is active. Spill provides the mechanism that maintains the upper bound.

\subsection{Dynamic programming principle and HJB equation}
For a short interval $h>0$, the dynamic programming principle gives
\begin{multline}
 V(t,s,q)=\inf_{a}\E_{t,s,q}\Bigg[
 \int_t^{t+h}e^{-\rho(r-t)}\ell(r,S_r,Q_r,a_r)\dif r +e^{-\rho h}V(t+h,S_{t+h},Q_{t+h})\Bigg].
 \label{eq:app-dpp}
\end{multline}
For a smooth test function $V\in C^{1,1,2}$, It\^o's formula for the discounted
process $e^{-\rho(r-t)}V(r,S_r,Q_r)$ gives
\begin{align}
 \dif\big(e^{-\rho(r-t)}V\big)
 =e^{-\rho(r-t)}\Big[&V_t+(Q_r-u_r-w_r)V_s
 +\kappa(\theta(r)-Q_r)V_q \\
 &+\tfrac12\sigma_0^2Q_rV_{qq}-\rho V\Big]\dif r
 +\dif M_r,
 \label{eq:app-discounted-ito}
\end{align}
where $M_r$ is a local martingale. Taking conditional expectations in
\eqref{eq:app-dpp}, dividing by $h$ and sending $h\downarrow0$ yields
\begin{equation}
 0=V_t+\min_{(u,w)\in\mathcal U(s,q)}\left\{
 \ell+(q-u-w)V_s+\kappa(\theta(t)-q)V_q
 +\tfrac12\sigma_0^2qV_{qq}-\rho V\right\},
 \label{eq:app-hjb-derived}
\end{equation}
which is equivalent to \eqref{eq:hjb}. When classical smoothness fails at storage
or inflow boundaries, the same dynamic-programming equation is interpreted in the
state-constrained viscosity sense \cite{fleming2006,crandall1992,soner1986}.

\begin{proposition}[Verification principle]\label{prop:verification}
Suppose a sufficiently regular function $V$ satisfies
\eqref{eq:app-hjb-derived}, the terminal condition $V(T,s,q)=g(s)$ and the relevant
state constraints. Then $V(t,s,q)\le J^a(t,s,q)$ for every admissible control $a$.
If an admissible measurable selector $a^*$ attains the Hamiltonian minimum almost
everywhere, then $V=J^{a^*}$ and $a^*$ is optimal.
\end{proposition}

\begin{proof}
For any admissible $a$, the HJB minimum implies
\begin{equation}
 V_t+\mathcal L^aV-\rho V+\ell^a\ge0,
 \label{eq:app-verification-ineq}
\end{equation}
where $\mathcal L^a$ contains the controlled storage drift and the inflow generator.
Integrate \eqref{eq:app-discounted-ito} from $t$ to $T$, localize the martingale if
necessary, take expectations, and use $V(T,S_T,Q_T)=g(S_T)$. Equation
\eqref{eq:app-verification-ineq} gives $J^a-V(t,s,q)\ge0$. If $a^*$ attains the
minimum, the inequality becomes equality along its state path, hence
$J^{a^*}=V$.
\end{proof}

\subsection{Full KKT derivation of the release threshold}
At a fixed state, let $\bar u(t,s,q)$ be the largest release consistent with turbine
and state feasibility. Terms independent of $u$ can be suppressed and the Hamiltonian
subproblem becomes
\begin{equation}
 \min_{0\le u\le\bar u}\ H(u),\qquad
 H(u)=\ell(t,s,q,u,w)-uV_s+\text{const}.
 \label{eq:app-hamiltonian-u}
\end{equation}
Introduce multipliers $\mu_-\ge0$ for $-u\le0$ and $\mu_+\ge0$ for
$u-\bar u\le0$. With $m(u)=-\ell_u$ and $\lambda_W=-V_s$, the KKT conditions are
\begin{align}
 -m(u)+\lambda_W-\mu_-+\mu_+&=0,
 \label{eq:app-kkt-stationarity}\\
 \mu_-u=0,\qquad \mu_+(u-\bar u)&=0,
 \label{eq:app-kkt-comp}\\
 0\le u\le\bar u,\qquad \mu_-,\mu_+&\ge0.
 \label{eq:app-kkt-feas}
\end{align}
If $0<u<\bar u$, complementary slackness gives $\mu_-=\mu_+=0$ and therefore
$m(u)=\lambda_W$. At $u=0$, stationarity requires $m(0)\le\lambda_W$; at
$u=\bar u$, it requires $m(\bar u)\ge\lambda_W$. These are precisely the three
cases stated in Proposition~\ref{prop:threshold}. Convexity of $\ell$ makes $H$
convex, so the KKT conditions are sufficient.

A useful analytic specialization is the convex thermal-shortfall cost
\begin{equation}
 \ell(u)=c_1x(u)+\tfrac12c_2x(u)^2+\ell_0,\qquad
 x(u)=\big(D(t)-\eta u\big)_+,
 \label{eq:app-quadratic-cost}
\end{equation}
with $c_1,c_2>0$ and conversion efficiency $\eta>0$. In the unsaturated region
$u<D(t)/\eta$,
\begin{equation}
 m(u)=\eta\left[c_1+c_2\{D(t)-\eta u\}\right].
 \label{eq:app-marginal-cost}
\end{equation}
The interior release satisfying $m(u)=\lambda_W$ is therefore
\begin{equation}
 u_{\mathrm{int}}(t,s,q)=
 \frac{D(t)}{\eta}
 -\frac{\lambda_W/\eta-c_1}{c_2\eta}.
 \label{eq:app-explicit-release}
\end{equation}
Consequently the optimal release in this specialization is the feasible projection
of \eqref{eq:app-explicit-release} onto
$[0,\min\{\bar u,D(t)/\eta\}]$, with the boundary cases selected by the marginal
conditions above. This formula makes the economics explicit: a larger water value
reduces current release, whereas a larger contemporaneous marginal thermal cost
increases it.

\section{SDDP cuts, dual sensitivity, and the discrete water value}
\label{app:sddp}
The central SDDP identity can be obtained directly from parametric linear-programming
duality. For a fixed stage $t$, scenario $k$, and incoming storage $s$, write the
successor recourse problem schematically as
\begin{equation}
 \Phi_{tk}(s)=\min_{x}\left\{c_{tk}^{\top}x
 +\gamma\widehat Q_{t+1}(s'):\
 s'-s-q_{tk}+u+w=0,\;x\in\mathcal X_{tk}\right\}.
 \label{eq:app-parametric-lp}
\end{equation}
Let $\pi_{tk}$ denote the dual multiplier of the displayed storage-balance equality,
using exactly the sign convention in \eqref{eq:app-parametric-lp}.

\begin{proposition}[Dual multiplier as a storage subgradient]\label{prop:app-dual}
Assume strong duality holds for \eqref{eq:app-parametric-lp} at a sampled state
$\bar s$. Then
\begin{equation}
 -\pi_{tk}\in\partial\Phi_{tk}(\bar s).
 \label{eq:app-dual-subgrad}
\end{equation}
For a probability-weighted stage value
$Q_t(s)=\sum_kp_k\Phi_{tk}(s)$, a valid subgradient is
\begin{equation}
 \beta_t=-\sum_kp_k\pi_{tk},
 \label{eq:app-beta}
\end{equation}
and the associated Benders cut is
\begin{equation}
 Q_t(s)\ge Q_t(\bar s)+\beta_t(s-\bar s)
 =\alpha_t+\beta_t s.
 \label{eq:app-benders-cut}
\end{equation}
Hence the positive marginal water value under this sign convention is
\begin{equation}
 \lambda_t^W=-\beta_t=\sum_kp_k\pi_{tk}.
 \label{eq:app-discrete-water}
\end{equation}
\end{proposition}

\begin{proof}
The Lagrangian term associated with the storage balance is
$\pi_{tk}(s'-s-q_{tk}+u+w)$. For the optimal multiplier
$\pi_{tk}^*$ at $\bar s$, the dual function evaluated at another parameter $s$
has the affine dependence
\begin{equation}
 d(\pi_{tk}^*;s)=d(\pi_{tk}^*;\bar s)
 -\pi_{tk}^*(s-\bar s).
 \end{equation}
Weak duality gives $\Phi_{tk}(s)\ge d(\pi_{tk}^*;s)$, while strong duality at
$\bar s$ gives $d(\pi_{tk}^*;\bar s)=\Phi_{tk}(\bar s)$. Therefore
\begin{equation}
 \Phi_{tk}(s)\ge\Phi_{tk}(\bar s)
 -\pi_{tk}^*(s-\bar s),
\end{equation}
which proves \eqref{eq:app-dual-subgrad}. Probability-weighted summation yields
\eqref{eq:app-beta}--\eqref{eq:app-benders-cut}. The water value is the reduction
in future cost per additional unit of storage, so it is the negative cut slope,
which gives \eqref{eq:app-discrete-water}. If several cuts are active, the stage
value is nondifferentiable and the admissible marginal values are the negatives of
the active subgradients. This is the discrete analogue of a nonsmooth HJB storage
derivative.
\end{proof}

The continuous and discrete constructions therefore estimate the same economic
object in different mathematical representations:
\begin{equation}
 \lambda_W^{\mathrm{HJB}}(t,s,q)=-V_s(t,s,q),\qquad
 \lambda_{t}^{\mathrm{SDDP}}=-\beta_t.
 \label{eq:app-bridge}
\end{equation}
Agreement is expected only after clocks, information sets, inflow laws, terminal
conditions and operating costs have been harmonized; equality of native stage
sequences is neither required nor implied.

\section{Haar storage projection bound}
\label{app:haar}
Let $I_j$ denote the $M=2^J$ equal cells of $[0,S_{\max}]$, each of length
$h=S_{\max}/M$, and let $P_Jf$ be the cell-average projection. For
$f\in H^1(0,S_{\max})$, the one-dimensional Poincar\'e inequality on each cell gives
\begin{equation}
 \|f-(f)_{I_j}\|_{L^2(I_j)}\le C h\|f'\|_{L^2(I_j)},
 \label{eq:app-poincare}
\end{equation}
where $(f)_{I_j}$ is the cell average and $C$ is independent of $j$ and $h$.
Squaring \eqref{eq:app-poincare} and summing over all cells yields
\begin{equation}
 \|f-P_Jf\|_{L^2(0,S_{\max})}^2
 \le C^2h^2\|f'\|_{L^2(0,S_{\max})}^2.
 \label{eq:app-haar-squared}
\end{equation}
Since $h=S_{\max}/M$, taking square roots gives
\begin{equation}
 \|f-P_Jf\|_{L^2(0,S_{\max})}
 \le C S_{\max}M^{-1}\|f'\|_{L^2(0,S_{\max})},
\end{equation}
which proves Theorem~\ref{thm:haar}. This estimate concerns only the storage
projection. It does not by itself establish first-order convergence of the complete
state-constrained HJB solver, whose error also contains time discretization, inflow
discretization, boundary treatment, interpolation and policy-iteration components.


\begin{thebibliography}{50}
\bibitem{labadie2004} J.~W. Labadie, ``Optimal operation of multireservoir
systems: State-of-the-art review,'' \emph{J. Water Resour. Plann. Manage.},
vol.~130, no.~2, pp.~93--111, 2004.

\bibitem{gjelsvik2010} A.~Gjelsvik, B.~Mo, and A.~Haugstad, ``Long- and
medium-term operations planning and stochastic modelling in hydro-dominated
power systems based on stochastic dual dynamic programming,'' in \emph{Handbook
of Power Systems I}. Berlin, Germany: Springer, 2010, pp.~33--55.

\bibitem{dequeiroz2016} A.~R. de Queiroz, ``Stochastic hydro-thermal
scheduling optimization: An overview,'' \emph{Renew. Sustain. Energy Rev.},
vol.~62, pp.~382--395, 2016, doi:10.1016/j.rser.2016.04.065.

\bibitem{heitsch2003} H.~Heitsch and W.~R\"omisch, ``Scenario reduction
algorithms in stochastic programming,'' \emph{Comput. Optim. Appl.}, vol.~24,
pp.~187--206, 2003.

\bibitem{kaut2007} M.~Kaut and S.~W. Wallace, ``Evaluation of scenario generation
methods for stochastic programming,'' \emph{Pacific J. Optim.}, vol.~3, no.~2,
pp.~257--271, 2007.

\bibitem{gneiting2007} T.~Gneiting and A.~E. Raftery, ``Strictly proper scoring
rules, prediction, and estimation,'' \emph{J. Amer. Stat. Assoc.}, vol.~102,
no.~477, pp.~359--378, 2007.

\bibitem{politis1994} D.~N. Politis and J.~P. Romano, ``The stationary
bootstrap,'' \emph{J. Amer. Stat. Assoc.}, vol.~89, no.~428, pp.~1303--1313,
1994.

\bibitem{labat2005} D.~Labat, ``Recent advances in wavelet analyses:
Part 1. A review of concepts,'' \emph{J. Hydrol.}, vol.~314, nos.~1--4,
pp.~275--288, 2005, doi:10.1016/j.jhydrol.2005.04.003.

\bibitem{sang2013} Y.-F. Sang, ``A review on the applications of
wavelet transform in hydrology time series analysis,'' \emph{Atmos. Res.},
vol.~122, pp.~8--15, 2013, doi:10.1016/j.atmosres.2012.11.003.

\bibitem{bodian2020} A.~Bodian, L.~Diop, G.~Panthou, H.~Dacosta,
A.~Deme, A.~Dezetter, P.~M. Ndiaye, I.~Diouf, and T.~Vischel, ``Recent trend
in hydroclimatic conditions in the Senegal River Basin,'' \emph{Water}, vol.~12,
no.~2, Art.~436, 2020, doi:10.3390/w12020436.

\bibitem{bruckmann2022} L.~Bruckmann, N.~Delbart, L.~Descroix,
and A.~Bodian, ``Recent hydrological evolutions of the Senegal River flood
(West Africa),'' \emph{Hydrol. Sci. J.}, vol.~67, no.~3, pp.~385--400, 2022,
doi:10.1080/02626667.2021.1998511.

\bibitem{raso2020} L.~Raso, J.-C. Bader, and S.~Weijs, ``Reservoir
operation optimized for hydropower production reduces conflict with traditional
water uses in the Senegal River,'' \emph{J. Water Resour. Plann. Manage.},
vol.~146, no.~4, Art.~05020003, 2020,
doi:10.1061/(ASCE)WR.1943-5452.0001076.

\bibitem{affognon2026wavelet}
S.~B. Affognon, B.~M. Ndiaye, P.~Mendy, and C.~M.~F. Kebe,
``From daily fluctuations to annual hydrological cycles: A wavelet-based analysis
of nonstationary seasonality in Senegal River hydropower inflows,''
\emph{arXiv preprint arXiv:2608.23470 [stat.AP]}, 2026.
doi:10.48550/arXiv.2608.23470.

\bibitem{cir1985} J.~C. Cox, J.~E. Ingersoll, and S.~A. Ross, ``A theory of the
term structure of interest rates,'' \emph{Econometrica}, vol.~53, no.~2,
pp.~385--407, 1985.

\bibitem{feller1951} W.~Feller, ``Two singular diffusion problems,'' \emph{Ann.
Math.}, vol.~54, no.~1, pp.~173--182, 1951.

\bibitem{alfonsi2005} A.~Alfonsi, ``On the discretization schemes for the CIR
(and Bessel squared) processes,'' \emph{Monte Carlo Methods Appl.}, vol.~11,
no.~4, pp.~355--384, 2005.

\bibitem{lord2010} R.~Lord, R.~Koekkoek, and D.~van Dijk, ``A
comparison of biased simulation schemes for stochastic volatility models,''
\emph{Quant. Finance}, vol.~10, no.~2, pp.~177--194, 2010,
doi:10.1080/14697680802392496.

\bibitem{pereira1991} M.~V.~F. Pereira and L.~M.~V.~G. Pinto, ``Multi-stage
stochastic optimization applied to energy planning,'' \emph{Math. Program.},
vol.~52, pp.~359--375, 1991.

\bibitem{philpott2008} A.~B. Philpott and Z.~Guan, ``On the convergence of
stochastic dual dynamic programming and related methods,'' \emph{Oper. Res.
Lett.}, vol.~36, no.~4, pp.~450--455, 2008.

\bibitem{shapiro2011} A.~Shapiro, ``Analysis of stochastic dual dynamic
programming method,'' \emph{Eur. J. Oper. Res.}, vol.~209, no.~1, pp.~63--72,
2011.

\bibitem{fullner2025} C.~F\"ullner and S.~Rebennack, ``Stochastic
dual dynamic programming and its variants: A review,'' \emph{SIAM Rev.},
vol.~67, no.~3, pp.~415--539, 2025, doi:10.1137/23M1575093.

\bibitem{philpottmatos2012} A.~B. Philpott and V.~L. de Matos,
``Dynamic sampling algorithms for multi-stage stochastic programs with risk
aversion,'' \emph{Eur. J. Oper. Res.}, vol.~218, no.~2, pp.~470--483, 2012,
doi:10.1016/j.ejor.2011.10.056.

\bibitem{dowson2021} O.~Dowson and L.~Kapelevich, ``SDDP.jl: A Julia
package for stochastic dual dynamic programming,'' \emph{INFORMS J. Comput.},
vol.~33, no.~1, pp.~27--33, 2021, doi:10.1287/ijoc.2020.0987.

\bibitem{fleming2006} W.~H. Fleming and H.~M. Soner, \emph{Controlled Markov
Processes and Viscosity Solutions}, 2nd~ed. New York, NY, USA: Springer, 2006.

\bibitem{kushner2001} H.~J. Kushner and P.~G. Dupuis, \emph{Numerical Methods for
Stochastic Control Problems in Continuous Time}, 2nd~ed. New York, NY, USA:
Springer, 2001.

\bibitem{crandall1992} M.~G. Crandall, H.~Ishii, and P.-L. Lions,
``User's guide to viscosity solutions of second order partial differential
equations,'' \emph{Bull. Amer. Math. Soc.}, vol.~27, no.~1, pp.~1--67, 1992,
doi:10.1090/S0273-0979-1992-00266-5.

\bibitem{soner1986} H.~M. Soner, ``Optimal control with state-space
constraint I,'' \emph{SIAM J. Control Optim.}, vol.~24, no.~3, pp.~552--561,
1986, doi:10.1137/0324032.

\bibitem{affognon2026inflow}
S.~B. Affognon, B.~M. Ndiaye, P.~Mendy, and C.~M.~F. Kebe,
``InFlow: entropic stochastic dual dynamic programming with HJB
cross-certification for seasonal energy storage,''
\emph{arXiv preprint arXiv:2609.11339 [math.OC]}, 2026.

\bibitem{beulertz2019} D.~Beulertz, S.~Charousset, D.~Most, S.~Giannelos, and
I.~Yueksel-Erguen, ``Development of a modular framework for future energy system
analysis,'' in \emph{Proc. 54th Int. Universities Power Engineering Conf.
(UPEC)}, 2019, pp.~1--6.

\bibitem{mallat2009} S.~Mallat, \emph{A Wavelet Tour of Signal Processing: The
Sparse Way}, 3rd~ed. Burlington, MA, USA: Academic Press, 2009.

\bibitem{daubechies1992} I.~Daubechies, \emph{Ten Lectures on Wavelets}.
Philadelphia, PA, USA: SIAM, 1992.

\bibitem{benders1962} J.~F. Benders, ``Partitioning procedures for solving
mixed-variables programming problems,'' \emph{Numer. Math.}, vol.~4,
pp.~238--252, 1962.

\bibitem{huangfu2018} Q.~Huangfu and J.~A.~J. Hall, ``Parallelizing the dual
revised simplex method,'' \emph{Math. Program. Comput.}, vol.~10, no.~1,
pp.~119--142, 2018.

\bibitem{efron1993} B.~Efron and R.~J. Tibshirani, \emph{An Introduction to the
Bootstrap}. New York, NY, USA: Chapman \& Hall, 1993.

\bibitem{applebaum2009} D.~Applebaum, \emph{L\'evy Processes and
Stochastic Calculus}, 2nd~ed. Cambridge, U.K.: Cambridge University Press,
2009.

\bibitem{katz2002} R.~W. Katz, M.~B. Parlange, and P.~Naveau, ``Statistics of
extremes in hydrology,'' \emph{Adv. Water Resour.}, vol.~25, nos.~8--12,
pp.~1287--1304, 2002.

\bibitem{loeffler2017} K.~L\"offler, K.~Hainsch, T.~Burandt,
P.-Y. Oei, C.~Kemfert, and C.~von Hirschhausen, ``Designing a model for the
global energy system---GENeSYS-MOD: An application of the open-source energy
modeling system (OSeMOSYS),'' \emph{Energies}, vol.~10, no.~10, Art.~1468,
2017, doi:10.3390/en10101468.

\bibitem{opentepes2022} A.~Ramos, E.~F. Alvarez, and S.~Lumbreras,
``OpenTEPES: Open-source transmission and generation expansion planning,''
\emph{SoftwareX}, vol.~18, Art.~101070, 2022,
doi:10.1016/j.softx.2022.101070.

\bibitem{loucks2017} D.~P. Loucks and E.~van Beek, \emph{Water Resource Systems
Planning and Management: An Introduction to Methods, Models, and Applications}.
Cham, Switzerland: Springer, 2017.

\end{thebibliography}
\end{document}